\documentclass[11pt]{amsart}
\usepackage[T1]{fontenc}
\usepackage{lmodern,amsmath,amssymb,amsthm,mathtools,microtype,needspace}
\usepackage[a4paper,margin=29mm]{geometry}
\usepackage{xcolor}
\usepackage[colorlinks=true,linkcolor=blue!45!black,citecolor=blue!45!black,urlcolor=blue!45!black]{hyperref}
\makeatletter
\@namedef{subjclassname@2020}{\textup{2020} Mathematics Subject Classification}
\makeatother
\newtheorem{theorem}{Theorem}[section]
\newtheorem{proposition}[theorem]{Proposition}
\newtheorem{lemma}[theorem]{Lemma}
\newtheorem{corollary}[theorem]{Corollary}
\theoremstyle{definition}
\newtheorem{example}[theorem]{Example}
\theoremstyle{remark}
\newtheorem{remark}[theorem]{Remark}

\newcommand{\Max}{\operatorname{Max}}
\newcommand{\Spec}{\operatorname{Spec}}
\newcommand{\Jac}{\operatorname{Jac}}
\newcommand{\Frac}{\operatorname{Frac}}
\newcommand{\Char}{\operatorname{char}}
\newcommand{\wgd}{\operatorname{w.gl.dim}}
\newcommand{\gld}{\operatorname{gl.dim}}
\newcommand{\fd}{\operatorname{fd}}
\newcommand{\Tor}{\operatorname{Tor}}
\newcommand{\lex}{\mathrm{lex}}
\newcommand{\Z}{\mathbb Z}
\newcommand{\Q}{\mathbb Q}
\newcommand{\F}{\mathbb F}
\newcommand{\N}{\mathbb N}

\title[Principal quotients of domains]{Von Neumann regular and semisimple rings realized  as principal quotients of domains}
\author{Xiaolei Zhang}
\address{School of Mathematics and Statistics, Tianshui Normal University, Tianshui 741000, China}
\email{zxlrghj@163.com}
\author{Wei Qi}
\address{School of Mathematics and Statistics, Tianshui Normal University, Tianshui 741000, China}
\email{qwlrghj@126.com}
\date{}
\hypersetup{pdftitle={Von Neumann regular and semisimple rings realized as principal quotients of domains},pdfauthor={}}

\begin{document}
	
	\begin{abstract}
	In this paper, we show that every commutative von Neumann regular ring is a quotient of a characteristic-zero B\'ezout domain by its principal Jacobson radical.   We then characterize the commutative semisimple rings that occur as quotients of Dedekind domains. For a finite product of fields $R=\prod_{i=1}^n k_i$, such a presentation exists exactly when the cardinals $|k_i|^{\aleph_0}$ are all equal. When this condition holds, the source may be taken to be a characteristic-zero semilocal principal ideal domain, and the quotient ideal may again be its Jacobson radical. 
	\end{abstract}
	
	\maketitle
	
	\section{Introduction}
	
Throughout this paper,	all rings are commutative with identity.	It is well-known that every commutative ring is a quotient of an integral domain by a principal ideal, and the source can even be taken to have characteristic zero \cite[Proposition 1.2]{Costa}.  However, it is a very difficult problem to realize every ring in a certain class of commutative rings  as a quotient ring of some integral domain.	
	The quotient problem for Pr\"ufer domains has a long history. In the terminology for rings with zero divisors, a Pr\"ufer ring is one whose finitely generated regular ideals are invertible. Boisen and Larsen \cite{BoisenLarsen} showed that a homomorphic image of a Pr\"ufer domain is a Pr\"ufer ring, while the converse fails. They also reduced the representation problem to the total quotient ring: a Pr\"ufer ring $A$ is an image of a Pr\"ufer domain if and only if its total quotient ring $Q(A)$ is such an image. Here $Q(A)$ denotes localization at all nonzerodivisors. This reduction concerns the existence of a Pr\"ufer source; it does not prescribe a principal kernel.
	
	Boisen and Sheldon \cite{BoisenSheldon} described the valuation structure of a presentation $A=D/I$ with $D$ Pr\"ufer. Their valuations on the target are understood in the sense of Manis \cite{Manis}. They obtained a bijection between Manis valuation pairs over $A$ and valuation overrings of $D$ centered on primes containing $I$. Under this correspondence, if a target element has finite value, then all of its lifts have the same value at the corresponding source valuation. The target value group is identified with an isolated, or convex, subgroup of the source value group. This extends the familiar prime-ideal correspondence and records how values pass to the quotient. The result applies after a presentation has been given; existence with additional conditions on the source and the kernel is a separate question.
	
	Related work studies the quotients themselves and the valuation intersections behind their presentations. Chuchel and Eggert \cite{ChuchelEggert} analyze the complete rings of quotients of images of semilocal Pr\"ufer domains through local quotients and topological completion. Using an irredundant primary decomposition of the kernel, they also describe the elements of the complete ring of quotients. Doering and Lequain \cite{DoeringLequain} prove a weak approximation theorem for valuations that need not be independent, and they study finite intersections with prescribed divisibility groups and residue fields. In the finite-spectrum construction below, explicit separating elements are enough for the ideal calculation.
	
	The local form of the existence problem is already restrictive. A \emph{chain ring} has its ideals linearly ordered by inclusion, and a chain domain is a valuation domain. Fuchs and Shelah \cite{FuchsShelah} proved in ZFC that some chain rings are not homomorphic images of valuation domains. Their examples use nonstandard divisible uniserial modules. Thus the ordering of the ideals of the target does not by itself provide a valuation-domain source.
	
	A different realization theorem is due to the first author of this paper \cite[Theorem~1.1]{Zhang}. By Bass's characterization \cite[Theorem~P]{Bass}, in the commutative case a perfect ring is a finite product of local rings whose maximal ideals are $T$-nilpotent; $T$-nilpotence means that every sequence in the ideal has a finite initial product equal to zero. Following Bazzoni and Salce \cite{BazzoniSalce}, a domain is \emph{almost perfect} if every nonzero proper quotient is perfect. Zhang shows that every perfect ring $B$ is $A/I$ for a characteristic-zero, one-dimensional semilocal almost perfect domain $A$, with $|\Max A|=|\Max B|$. The ideal $I$ contains a nonzero $g\in\Jac(A)$ for which $A/gA$ is perfect. The source can be chosen with
	\[
	\Spec A=\{(0),\mathfrak n_1,\ldots,\mathfrak n_r\},
	\qquad \Frac(A)=A[1/g].
	\]
	This gives the local realization and characterizes perfect rings as proper quotients of almost perfect domains. It permits mixed residue characteristics and non-nilpotent radicals, and the same source realizes the quotients of $B$ \cite[Sections~8--9]{Zhang}. Here $B$ is a quotient of the perfect principal quotient $A/gA$; the kernel defining $B$ need not be principal.
	
	For principal quotients of B\'ezout domains, several structural and homological properties are known. In particular, Sorokin \cite[Theorems~4--5]{Sorokin} shows that if $A$ is a B\'ezout duo-domain and $a\in A$ is a nonzero nonunit, then $A/aA$ has weak global dimension either zero or infinity; finite weak global dimension is equivalent to von Neumann regularity. These results start with a given principal quotient. Theorem~\ref{mainA} constructs a source for a prescribed absolutely flat target.
	
	Our first result gives a B\'ezout source for every absolutely flat target. Recall that a domain is B\'ezout if every finitely generated ideal is principal. We use the standard ideal theory of B\'ezout and Pr\"ufer domains as presented in \cite{Gilmer,FuchsSalce}.
	
	\begin{theorem}\label{mainA}
		Let $R$ be a commutative von Neumann regular ring. Then there are a characteristic-zero B\'ezout domain $D$ and a nonzero nonunit $t\in D$ such that
		\begin{equation}\label{main-eq}
			D/tD\cong R,\qquad \Jac(D)=tD,\qquad \wgd D=1.
		\end{equation}
		The construction can be made with
		\[
		\Frac(D)=\Q(T_a:a\in R),\qquad |D|=\max\{\aleph_0,|R|\},
		\]
		where the $T_a$ are algebraically independent.
	\end{theorem}
	
	The residue fields of $R$ may have different characteristics and arbitrary transcendence degrees. Moreover, $R$ need not be the full product of these fields. The proof uses every point of the Boolean space $\Max R$. At each point, translated monomial valuations realize the required residue field; a preliminary Gauss valuation handles positive residue characteristic in a common characteristic-zero field. For each rational function, its values are constant on the parts of a finite clopen partition. This allows the residue tuples to be glued in $R$ and supplies generators of two-generated ideals. A common least positive value then makes the residue kernel principal.
	
	For semisimple targets, one can ask exactly when the source may be Dedekind. The necessary cardinal estimates belong to a more general theory. Kearnes and Oman \cite[Lemma~2.1]{KearnesOman} proved that a Noetherian domain $A$ which is not a finite field satisfies
	\[
	\max\{\aleph_0,|A/J|\}\le |A|\le |A/J|^{\aleph_0}
	\]
	for every proper ideal $J$. Their Theorem~2.6 realizes admissible finite lists of residue-field \emph{cardinalities} in one-dimensional semilocal Noetherian domains. Our second result concerns specified fields, with their given field structures, and produces a principal ideal domain.
	
	\Needspace{19\baselineskip}
	\begin{theorem}\label{mainB}
		Let $R=\prod_{i=1}^n k_i$ be a commutative semisimple ring, and put $\kappa_i=|k_i|$. The following conditions are equivalent.
		\begin{enumerate}
			\item $R$ is a quotient of a Dedekind domain by an ideal.
			\item There are a characteristic-zero semilocal PID $D$ and a nonzero nonunit $t\in D$ such that
			\[
			D/tD\cong R,\qquad tD=\Jac(D),\qquad |\Max D|=n.
			\]
			\item $\kappa_1^{\aleph_0}=\cdots=\kappa_n^{\aleph_0}$.
			\item $\kappa_j\le\kappa_i^{\aleph_0}$ for all $i,j$.
		\end{enumerate}
		If these conditions hold, put $M=\kappa_i^{\aleph_0}$ and
		$\mu=\max\{\aleph_0,\kappa_1,\ldots,\kappa_n\}$. For every cardinal $\lambda$ with $\mu\le\lambda\le M$, the source in (2) can be chosen so that
		\[
		|D|=\lambda,\qquad \Frac(D)=\Q(U_\alpha:\alpha<\lambda),
		\]
		where the $U_\alpha$ are algebraically independent.
	\end{theorem}
	
	In particular, any finite product of fields of cardinality at most $2^{\aleph_0}$ has a characteristic-zero semilocal PID source, regardless of the residue characteristics. If one factor is finite, then the same cardinal bound on all other factors is also necessary. Taking
	\[
	R=\F_2\times\F_2(X_s:s\in S),\qquad |S|>2^{\aleph_0},
	\]
	gives a semisimple ring with no Dedekind-domain source, even when the kernel is allowed to be arbitrary. The exclusion of a Noetherian source also follows from the bound of Kearnes and Oman. The ring still has the B\'ezout source of Theorem~\ref{mainA}. Thus weak global dimension one can always be achieved, while global dimension at most one cannot; the latter assertion is justified in Corollary~\ref{negative-example}.
	
	The construction for Theorem~\ref{mainB} uses complete DVRs of characteristic zero with the specified residue fields. Every prescribed residue class contains enough elements to choose algebraically independent lifts of a common set of variables. Additional variables have value one at one chosen place and zero at all the others. These elements give a direct proof that the intersection is a PID. In particular, surjectivity onto the product is built into the lifts, rather than inferred from surjectivity onto separate factors.
	
	Section~\ref{sec-regular} proves Theorem~\ref{mainA}, describes the maximal localizations, and considers higher powers of the principal kernel. Section~\ref{sec-semisimple} proves Theorem~\ref{mainB} and records its consequences. The two constructions serve different purposes: the first handles arbitrary Boolean spectra, while the second uses the finiteness of the spectrum to obtain discrete valuations and Noetherianity.
	
	\section{Realization of von Neumann regular rings}\label{sec-regular}
	
	\subsection{Principal quotients and Boolean spectra}
	
	We begin with an elementary construction which explains the unrestricted principal-quotient assertion in the introduction. Its proof is included in full; no B\'ezout or Pr\"ufer property is claimed for this source.
	
	\begin{proposition}\cite[Proposition 1.2]{Costa}\label{universal}
		Every commutative ring $B$ is isomorphic to $E/zE$ for a characteristic-zero integral domain $E$ and a nonzero $z\in E$. If $B\ne0$, then $z$ is a nonunit.
	\end{proposition}
	
	For the rest of the positive construction, $R$ is a fixed commutative von Neumann regular ring. If $a\in R$, write
	\[
	\mathcal D(a)=\{x\in\Max R:a\notin x\},\qquad
	\mathcal V(a)=\{x\in\Max R:a\in x\}.
	\]
	This notation distinguishes basic open sets from the source domain $D$ to be constructed.
	
	The following facts are standard in the theory of commutative regular rings. We include the elementary proof to fix the idempotents and generalized inverses used later.
	
	\begin{lemma}\cite{Pierce}\cite[Tags 092F and 00EE]{Stacks}\label{boolean}
		Put $X=\Max R$ and $k_x=R/x$. Then the following hold.
		\begin{enumerate}
			\item Every prime of $R$ is maximal, $R$ is reduced, and $R_x\cong k_x$ for each $x\in X$.
			\item The map $\iota:R\to\prod_{x\in X}k_x$ given by residues is injective.
			\item The space $X$ is compact Hausdorff with a basis of clopen sets. Each clopen set is $\mathcal D(e)$ for a unique idempotent $e\in R$.
			\item For each $b\in R$ there is a unique $b^\#\in R$ such that
			\[
			b^2b^\#=b,\qquad b(b^\#)^2=b^\#.
			\]
			Its residue in $k_x$ is $(b+x)^{-1}$ when $b\notin x$, and is zero when $b\in x$.
		\end{enumerate}
	\end{lemma}
	
	\begin{proof}
		The identity $b=b^2c$ passes to quotients and localizations. A domain satisfying it is a field, since $b\ne0$ implies $bc=1$. Thus every prime is maximal. Iterating $b=b^2c$ gives $b=b^n c^{n-1}$ for every $n\ge1$; a nilpotent $b$ must therefore be zero. A localization at a maximal ideal is reduced and has a unique prime. Its maximal ideal is its nilradical and hence is zero, so it is a field. The residue map identifies it with $k_x$. The kernel of $\iota$ is the intersection of all primes, which is zero.
		
		Given $b=b^2c$, put $e=bc$. Then $e^2=e$ and $bR=eR$, so $\mathcal D(b)=\mathcal D(e)$. For an idempotent $e$, the complement of $\mathcal D(e)$ is $\mathcal D(1-e)$. Thus these basic opens are clopen. If $x\ne y$, choose $b\in x\setminus y$ and use the complementary idempotent opens associated to $b$ to separate $x$ and $y$. Compactness is the quasi-compactness of $\Spec R=X$.
		
		A clopen set $U$ is compact, so finitely many idempotent basic opens $\mathcal D(e_i)\subseteq U$ cover it. Then
		\[
		e=1-\prod_i(1-e_i)
		\]
		is idempotent and $\mathcal D(e)=U$. An idempotent has residue either zero or one in every $k_x$, and injectivity of $\iota$ proves uniqueness.
		
		Finally, set $b^\#=bc^2$. The relation $b=b^2c$ verifies both displayed identities. In a field these identities determine exactly the residues stated in (4), including the zero case. Injectivity of $\iota$ proves uniqueness in $R$.
	\end{proof}
	
	In particular, the idempotent associated to a clopen set is its characteristic function in the residue fields. Finite clopen partitions give finite families of orthogonal idempotents summing to $1$. These elementary facts also explain the usual equivalence with absolute flatness; see \cite[Tag 092F]{Stacks}. Notice that $\Jac(R)=0$, since all primes are maximal and $R$ is reduced.
	
	\subsection{Valuations on a common rational function field}
	
	We use additive valuations and set $v(0)=\infty$. All ordered groups are totally ordered abelian groups. A valuation ring consists of the elements of nonnegative value, and its maximal ideal consists of those of positive value.
	
	\subsubsection{The ordered weights}
	
	Choose a well-order on $R\setminus\{0\}$. Let
	\[
	H=\bigoplus_{a\in R\setminus\{0\}}\Z e_a
	\]
	with the lexicographic order: the sign of a nonzero element is the sign of its first nonzero coordinate. Every element has finite support, so this is well defined. Put
	\begin{equation}\label{weights}
		\Gamma=H\oplus_{\lex}\Z,\qquad
		\gamma_a=(e_a,0)\ (a\ne0),\qquad \gamma_0=(0,1).
	\end{equation}
	The family $(\gamma_a)_{a\in R}$ is a basis of the underlying free abelian group. Therefore, if $\alpha:R\to\N$ has finite support, write $Y^\alpha=\prod_aY_a^{\alpha(a)}$ and define its weight by
	\[
	\gamma(\alpha)=\sum_a\alpha(a)\gamma_a.
	\]
	This determines $\alpha$ uniquely. All nonconstant monomials have positive weight. Moreover, $\gamma_0$ is the least positive element of $\Gamma$. Indeed, a positive element with nonzero $H$-component is larger than every element of $\{0\}\oplus\Z$; the other positive elements are $(0,n)$ with $n\ge1$.
	
	We use the usual monomial-valuation construction. General valuation background is in \cite{EnglerPrestel}; the proof below gives the specific ordered-group and residue calculations required here.
	
	\begin{lemma}\label{monomial}
		For any field $k$, the rule
		\[
		w\left(\sum_\alpha c_\alpha Y^\alpha\right)
		=\min_{c_\alpha\ne0}\gamma(\alpha)
		\]
		on nonzero polynomials extends to a valuation on $k(Y_a:a\in R)$. Its residue field is $k$. If nonzero polynomials $P,Q$ have lowest-weight exponents $\alpha,\beta$, then $w(P/Q)=0$ implies $\alpha=\beta$, and the residue of $P/Q$ is the ratio of their lowest-weight coefficients.
	\end{lemma}
	
	\begin{proof}
		The nonzero terms of a polynomial have distinct weights, so there is a unique lowest term. If the lowest exponents of $P,Q$ are $\alpha,\beta$, the product of the corresponding terms has weight $\gamma(\alpha)+\gamma(\beta)$. Every other product of terms has strictly larger weight. Thus its nonzero coefficient cannot cancel, and $w(PQ)=w(P)+w(Q)$. The inequality $w(P+Q)\ge\min\{w(P),w(Q)\}$ follows directly from the formula. Hence $w(P/Q)=w(P)-w(Q)$ is a well-defined valuation.
		
		If this value is zero, injectivity of the weight map gives $\alpha=\beta$. Let $c,d\in k^\times$ be the lowest coefficients of $P,Q$. The lowest terms in $dP-cQ$ cancel, so
		\[
		w\left(\frac P Q-\frac c d\right)>0
		\]
		unless the difference is zero, when its value is $\infty$. Every value-zero fraction therefore has residue $c/d$. Positive-value fractions have zero residue, while constants realize every element of $k$. This identifies the residue field.
	\end{proof}
	
	\subsubsection{Translation and residue characteristic}
	
	Set
	\begin{equation}\label{commonfield}
		K=\Q(T_a:a\in R),\qquad t=T_0,\qquad
		\Delta=\Z\oplus_{\lex}\Gamma,\qquad \delta=(0,\gamma_0).
	\end{equation}
	Here $T_0$ is a nonzero indeterminate. The element $\delta$ is the least positive element of $\Delta$.
	
	For $x\in X$, write $\bar a=a+x\in k_x$, and let $k_x^0$ be the prime subfield of $k_x$. Translation gives an embedding
	\begin{equation}\label{translation}
		k_x^0(T_a:a\in R)\longrightarrow k_x(Y_a:a\in R),
		\qquad T_a\longmapsto\bar a+Y_a.
	\end{equation}
	To see injectivity, first extend the source coefficient field to $k_x$; translation by $(\bar a)$ is an automorphism of the polynomial ring over $k_x$. Restrict the valuation in Lemma~\ref{monomial} along this embedding, and call it $\omega_x$. Its residue field is exactly $k_x$: all residues lie in $k_x$, and $T_a$ has residue $\bar a$, which runs through every element of $k_x$.
	
	If $\Char k_x=0$, the prime field is $\Q$, so define on $K$
	\begin{equation}\label{charzero}
		v_x(f)=(0,\omega_x(f)).
	\end{equation}
	There is no map $\Q\to k_x$ when $\Char k_x>0$. We handle that case by a separate first valuation.
	
	The next lemma is the Gauss extension of the $p$-adic valuation to a rational function field; see \cite{EnglerPrestel} for the general construction.
	
	\begin{lemma}\label{gauss}
		For every prime integer $p$, there is a valuation $u_p$ on $K$ given on nonzero polynomials by
		\[
		u_p\left(\sum_\alpha q_\alpha T^\alpha\right)
		=\min_{q_\alpha\ne0}\operatorname{ord}_p(q_\alpha).
		\]
		Its value group is $\Z$ and its residue field is $\F_p(T_a:a\in R)$. Every primitive polynomial in $\Z[T_a:a\in R]$ has $u_p$-value zero for every $p$.
	\end{lemma}
	
	\begin{proof}
		Multiply each polynomial by a suitable power of $p$ to make its minimum coefficient order zero. The normalized polynomial has coefficients in $\Z_{(p)}$ and nonzero reduction over $\F_p$. Products of nonzero polynomial reductions remain nonzero. Undoing the normalization proves multiplicativity of the formula; the sum inequality is immediate. It therefore extends to the fraction field. The element $p$ has value $1$. Reducing fractions of value zero identifies the residue field with the stated rational function field. A primitive integer polynomial has a coefficient not divisible by $p$, proving the last assertion. Each calculation involves only finitely many variables.
	\end{proof}
	
	Suppose $\Char k_x=p>0$. If $f\in K^\times$ and $n=u_p(f)$, then $p^{-n}f$ is a unit of the Gauss valuation ring and has a nonzero reduction, denoted $\widetilde{p^{-n}f}$. Define
	\begin{equation}\label{composite}
		v_x(f)=\bigl(n,\omega_x(\widetilde{p^{-n}f})\bigr)\in\Delta.
	\end{equation}
	Multiplicativity follows by multiplying normalized units and reducing. For a sum with different first values, the lower first value persists and its normalized reduction is unchanged. For a sum with equal first values, either the normalized reductions have nonzero sum, in which case the inequality for $\omega_x$ applies, or they cancel, in which case the first value increases. Thus \eqref{composite} is a valuation. This is the usual composition of valuations, written here with a fixed normalization; see also \cite{EnglerPrestel}.
	
	Let $V_x$ be the valuation ring of $v_x$, and let $\rho_x:V_x\to k_x$ be its residue map. In positive characteristic, an element with positive first value has residue zero. An element with first value zero is reduced first by $u_p$ and then by $\omega_x$. It follows that the residue field is again $k_x$. We regard all $v_x$ as taking values in the same ordered group $\Delta$; their images need not be all of $\Delta$.
	
	\begin{lemma}\label{variables}
		For $a\in R$ and $x\in X$,
		\[
		v_x(T_a)=
		\begin{cases}
			0,&a\notin x,\\
			(0,\gamma_a),&a\in x,
		\end{cases}
		\qquad \rho_x(T_a)=a+x.
		\]
		In particular, $v_x(t)=\delta$ for every $x$, and $\Z[T_a:a\in R]\subseteq V_x$.
	\end{lemma}
	
	\begin{proof}
		The translated polynomial is $\bar a+Y_a$. If $\bar a\ne0$, its lowest term is the constant term, giving value zero and residue $\bar a$. If $\bar a=0$, it is $Y_a$, giving the stated positive value and residue zero. In positive characteristic $u_p(T_a)=0$, so the same argument applies to its Gauss reduction. Integers have nonnegative value: their first value is nonnegative, and if it is zero their remaining value is zero. All the indicated polynomial generators therefore lie in $V_x$.
	\end{proof}
	
	\subsection{Finite clopen partitions and the residue ring}
	
	Individual valuation rings do not by themselves control their intersection. We now prove the uniformity that will do so in this construction.
	
	Every $f\in K^\times$ can be written
	\begin{equation}\label{primitiveform}
		f=cP/Q,\qquad c\in\Q^\times,\quad
		P,Q\in\Z[T_a:a\in R]\text{ nonzero and primitive}.
	\end{equation}
	Indeed, clear denominators in a numerator and a denominator, then take out their integer contents. For such a $P$, perform translation in the ring $R$ itself:
	\begin{equation}\label{coefficients}
		P(a+Y_a:a\in R)=\sum_\alpha A_\alpha(P)Y^\alpha
		\quad\text{in }R[Y_a:a\in R].
	\end{equation}
	The sum is finite. Only the finitely many variables occurring in $P$ are translated. In particular, these coefficients are elements of $R$, not choices of lifts to characteristic zero.
	
	\begin{lemma}\label{leadingpartition}
		At every $x\in X$, the reduction of \eqref{coefficients} in $k_x[Y_a:a\in R]$ is nonzero. Let $\alpha_P(x)$ be its lowest-weight exponent. For each possible exponent $\alpha$, the set
		\[
		U^P_\alpha=\{x\in X:\alpha_P(x)=\alpha\}
		\]
		is clopen. The nonempty sets $U^P_\alpha$ form a finite partition of $X$.
	\end{lemma}
	
	\begin{proof}
		In characteristic zero, a nonzero integer polynomial stays nonzero over the prime field. In characteristic $p$, primitivity ensures that its reduction is nonzero. Translation over the field $k_x$ preserves nonzeroness in either case. A unique lowest-weight term therefore exists.
		
		Let $\mathcal S$ be the finite set of exponents with $A_\alpha(P)\ne0$ in $R$. For $\alpha\in\mathcal S$,
		\[
		U^P_\alpha=\mathcal D(A_\alpha(P))\cap
		\bigcap_{\substack{\beta\in\mathcal S\\
				\gamma(\beta)<\gamma(\alpha)}}
		\mathcal V(A_\beta(P)).
		\]
		Lemma~\ref{boolean} makes each set on the right clopen. The existence and uniqueness of the lowest exponent give the partition assertion.
	\end{proof}
	
	Write $e^P_\alpha$ for the characteristic idempotent of $U^P_\alpha$, and put
	\[
	\lambda_P(x)=\gamma(\alpha_P(x)).
	\]
	An exponent absent from the expansion is assigned $A_\alpha(P)=e^P_\alpha=0$.
	
	\begin{proposition}\label{finitevalues}
		For each $f\in K^\times$, the function $x\mapsto v_x(f)$ has finite image and clopen fibers. In particular, for $f,g\in K^\times$ the sets
		\[
		\{x:v_x(f)\ge0\},\qquad \{x:v_x(f)\le v_x(g)\}
		\]
		are clopen.
	\end{proposition}
	
	\begin{proof}
		Use \eqref{primitiveform}, and define
		\[
		n_c(x)=
		\begin{cases}
			\operatorname{ord}_p(c),&\Char k_x=p>0,\\
			0,&\Char k_x=0.
		\end{cases}
		\]
		Only the finitely many prime divisors of the numerator and denominator of $c$ can give a nonzero value. The locus of characteristic $p$ is $\mathcal V(p1_R)$, which is clopen. Hence $n_c$ has finite image and clopen fibers; its zero fiber is the complement of a finite union of these loci. No openness of the characteristic-zero locus is being assumed.
		
		We claim that
		\begin{equation}\label{valueformula}
			v_x(f)=\bigl(n_c(x),\lambda_P(x)-\lambda_Q(x)\bigr).
		\end{equation}
		In characteristic zero this follows from translation and the fact that a nonzero rational constant has monomial value zero. In characteristic $p$, both $P$ and $Q$ have Gauss value zero. After normalizing $c$ by $p^{-\operatorname{ord}_p(c)}$, its reduction is a nonzero constant in $\F_p$, again of monomial value zero. The translated reductions of $P,Q$ have precisely the leading terms used to define $\lambda_P,\lambda_Q$. This proves the formula in all cases.
		
		The three functions on the right of \eqref{valueformula} have finite image and clopen fibers. Their common refinement is a finite clopen partition on which $v_x(f)$ is constant. Refining once more for $g$ proves the assertions about inequalities.
	\end{proof}
	
	Define
	\begin{equation}\label{intersection}
		D=\bigcap_{x\in X}V_x\subseteq K.
	\end{equation}
	It is a characteristic-zero domain, and Lemma~\ref{variables} gives
	\begin{equation}\label{fractionfield}
		\Z[T_a:a\in R]\subseteq D\subseteq K,
		\qquad \Frac(D)=K.
	\end{equation}
	The residue maps give a homomorphism $D\to\prod_x k_x$. The next proposition identifies its image exactly; this is stronger than surjectivity onto each separate residue field.
	
	\begin{proposition}\label{residuemap}
		There is a surjective homomorphism $\rho:D\twoheadrightarrow R$ such that
		\[
		\iota(\rho(f))=(\rho_x(f))_{x\in X},\qquad \rho(T_a)=a.
		\]
	\end{proposition}
	
	\begin{proof}
		Fix $f\in D\setminus\{0\}$ and write $f=cP/Q$ as in \eqref{primitiveform}. Write $c=r/s$ in lowest terms, with $r\in\Z\setminus\{0\}$ and $s>0$. If a prime $p\mid s$ occurred as a residue characteristic, then $n_c(x)<0$ at a point of that characteristic. The first component in \eqref{valueformula} would force $v_x(f)<0$, contrary to $f\in D$. Thus no prime factor of $s$ occurs as a residue characteristic. The element $s1_R$ belongs to no maximal ideal, so it is a unit in $R$. Set
		\[
		c_R=(r1_R)(s1_R)^{-1}\in R.
		\]
		Using the generalized inverse from Lemma~\ref{boolean}, define the finite sum
		\begin{equation}\label{residueformula}
			r_f=c_R\sum_\alpha e^P_\alpha e^Q_\alpha
			A_\alpha(P)\bigl(A_\alpha(Q)\bigr)^\#\in R.
		\end{equation}
		We verify its residues point by point.
		
		At a point where $n_c(x)>0$, the composite residue of $f$ is zero. Here the characteristic divides $r$ and not $s$, so the residue of $c_R$, and hence of $r_f$, is also zero. A negative $n_c(x)$ is excluded by integrality.
		
		Suppose $n_c(x)=0$. Then $\lambda_P(x)\ge\lambda_Q(x)$. If the inequality is strict, the residue of $f$ is zero. The lowest exponents of $P,Q$ differ, so $e^P_\alpha e^Q_\alpha$ vanishes at $x$ for every $\alpha$, and the residue of \eqref{residueformula} is zero. If equality holds, injectivity of $\gamma$ gives a common lowest exponent $\alpha$. Its coefficient in $Q$ is nonzero at $x$. The only surviving summand in \eqref{residueformula} is then the ratio of the two lowest coefficients, multiplied by the residue of $c_R$. In characteristic zero this is exactly Lemma~\ref{monomial}. In characteristic $p$ the equality $n_c(x)=0$ says that $r,s$ are both $p$-units, so reduction of $c$ is also exactly the residue of $c_R$. The same lemma applies after Gauss reduction. Thus $r_f+x=\rho_x(f)$ in every case.
		
		Set $r_0=0$. The product residue map therefore has image in $\iota(R)$. Since $\iota$ is injective, it defines a unique homomorphism $\rho:D\to R$, independent of every expression used in \eqref{primitiveform}. Lemma~\ref{variables} gives $\rho(T_a)=a$. Hence $\rho$ is onto.
	\end{proof}
	
	\subsection{The B\'ezout property and the principal radical}
	
	We now use the finite partitions in two different ways: first to construct a generator of a two-generated ideal, and then to identify the common residue kernel.
	
	\begin{proposition}\label{bezout}
		The domain $D$ in \eqref{intersection} is B\'ezout.
	\end{proposition}
	
	\begin{proof}
		Take nonzero $f,g\in D$. By Proposition~\ref{finitevalues}, the comparison set
		\[
		U=\{x\in X:v_x(f)\le v_x(g)\}
		\]
		is clopen. Let $e\in R$ be its characteristic idempotent, and put $h=T_{1-e}\in D$. At $x\in U$, the residue of $h$ is zero and $v_x(h)>0$. At $x\notin U$, its residue is one and $v_x(h)=0$.
		
		Set $b=f+hg$. On $U$ we have $v_x(hg)>v_x(g)\ge v_x(f)$, while off $U$ we have $v_x(hg)=v_x(g)<v_x(f)$. For a valuation, unequal values in a sum give the smaller value: this follows by applying the valuation inequality also to $f=(f+hg)-hg$. Therefore
		\begin{equation}\label{minimum}
			v_x(b)=\min\{v_x(f),v_x(g)\}\quad(x\in X).
		\end{equation}
		In particular $b\ne0$, since $X\ne\varnothing$. The formula implies $f/b,g/b\in V_x$ for every $x$, and hence $f/b,g/b\in D$. Thus $(f,g)D\subseteq bD$. The reverse inclusion follows from $b=f+hg$. If one generator is zero, the ideal is already principal. Induction now handles every finite set of generators.
	\end{proof}
	
	\begin{proposition}\label{kernel}
		For $t=T_0$, one has
		\[
		\ker\rho=tD=\Jac(D).
		\]
		Moreover, for $f\in D$, the element $f$ is a unit if and only if $\rho(f)$ is a unit in $R$.
	\end{proposition}
	
	\begin{proof}
		For nonzero $f\in D$, the equality $\rho(f)=0$ is equivalent, by injectivity of $\iota$, to $\rho_x(f)=0$ for every $x$. This means $v_x(f)>0$ for every $x$. Since $\delta$ is the least positive element of $\Delta$ and $v_x(t)=\delta$, it is equivalent to
		\[
		v_x(f/t)=v_x(f)-\delta\ge0\quad\text{for every }x.
		\]
		By definition of $D$, this means $f/t\in D$. Including $f=0$ gives $\ker\rho=tD$.
		
		For $d\in D$, the element $td$ has positive value everywhere unless $d=0$. Hence $1-td$ has value zero everywhere, is nonzero, and has inverse in every $V_x$. It is therefore a unit in $D$. The Jacobson-radical criterion \cite[Proposition~1.9]{AtiyahMacdonald} gives $tD\subseteq\Jac(D)$. Conversely, every maximal ideal of $D$ contains $tD$, so the maximal-ideal correspondence for $D/tD\cong R$ gives
		\[
		\Jac(D)/tD=\Jac(R)=0.
		\]
		Thus equality holds.
		
		If $\rho(f)$ is a unit, its residue at every $x$ is nonzero. It follows that $v_x(f)=0$ everywhere and $f^{-1}\in D$. The converse holds for any unital homomorphism.
	\end{proof}
	
	\begin{proof}[Proof of Theorem~\ref{mainA}]
		Propositions~\ref{residuemap}, \ref{bezout}, and \ref{kernel} establish the quotient and ideal assertions. The field containment \eqref{fractionfield} gives characteristic zero and the specified fraction field. Since $R\ne0$, the ideal $tD$ is proper, and $t$ is a nonzero nonunit.
		
		We explain the weak-dimension argument. If $\mathfrak n$ is a maximal ideal of a B\'ezout domain, every two-generated ideal in $D_{\mathfrak n}$ is principal. For nonzero $a,b$ write $(a,b)=cD_{\mathfrak n}$ and $a=cr$, $b=cs$. Cancellation gives $(r,s)=D_{\mathfrak n}$. In a local ring, one of $r,s$ is a unit. Thus $a$ divides $b$ or $b$ divides $a$. This proves that $D_{\mathfrak n}$ is a valuation domain.
		
		We use the standard fact that torsion-free modules over valuation domains are flat \cite[Tag 0539]{Stacks}; more generally, torsion-free modules over Pr\"ufer domains are flat \cite[Theorem~4.35]{Rotman}. Here is the relation calculation in our local setting. For a relation $\sum_{i=1}^q a_i m_i=0$ in a torsion-free $V$-module, choose a nonzero coefficient $a_1$ dividing the others and write $a_i=a_1b_i$ for $i\ge2$. Cancellation gives
		\[
		m_1=-\sum_{j=2}^q b_jm_j.
		\]
		For $j\ge2$, put $y_j=m_j$, $c_{1j}=-b_j$, and $c_{ij}=\delta_{ij}$ for $i\ge2$. Then
		\[
		m_i=\sum_{j=2}^q c_{ij}y_j,\qquad
		\sum_{i=1}^q a_ic_{ij}=-a_1b_j+a_j=0.
		\]
		These equations verify the equational criterion for flatness \cite[Tag 00HK]{Stacks}. A relation with all coefficients zero satisfies the criterion by taking the identity coefficient matrix. For a one-term relation with nonzero coefficient, the element is zero and the empty family suffices. Since localization preserves torsion-freeness, and flatness is detected at maximal ideals \cite[Tag 00HT]{Stacks}, every torsion-free $D$-module is flat.
		
		For any $D$-module $M$, a free presentation $0\to L\to F\to M\to0$ has torsion-free kernel $L$, because $D$ is a domain. Thus $L$ is flat and $\fd_D M\le1$. On the other hand, multiplication by $t$ gives a free resolution
		\[
		0\longrightarrow D\xrightarrow{\ t\ }D\longrightarrow D/tD\longrightarrow0.
		\]
		Tensoring with $D/tD$ shows that
		$\Tor^D_1(D/tD,D/tD)\cong D/tD\ne0$. Hence $\fd_D(D/tD)=1$ and $\wgd D=1$.
		
		Finally, both outer rings in \eqref{fractionfield} have cardinality $\max\{\aleph_0,|R|\}$: each of their elements uses only finitely many indices and coefficients, and the polynomial ring contains all $T_a$ and all integers. The intermediate ring $D$ has the same cardinality.
	\end{proof}
	
	\begin{remark}
		The two ends of $\Delta=\Z\oplus_{\lex}H\oplus_{\lex}\Z$ have separate roles. The first component handles positive residue characteristic; a negative value there cannot be offset by any monomial weight. The last component supplies the common least positive value, which is what makes the kernel principal. A principal maximal ideal in the resulting local rings does not imply that they are discrete valuation rings.
	\end{remark}
	
	\subsection{Local structure and further quotients}
	
	\begin{proposition}\label{localizations}
		For $x\in X$, put $M_x=\rho^{-1}(x)$. These are exactly the maximal ideals of $D$, and
		\[
		D_{M_x}=V_x,\qquad V_x/tV_x\cong k_x.
		\]
		The induced map $\Max D\to\Max R$ is a homeomorphism.
	\end{proposition}
	
	\begin{proof}
		Every maximal ideal of $D$ contains $tD=\Jac(D)$. The quotient $D/tD\cong R$ therefore gives the asserted bijection and homeomorphism of maximal spectra. For $d\in D$, one has $d\notin M_x$ exactly when $\rho_x(d)\ne0$, or equivalently $v_x(d)=0$. All such $d$ are units in $V_x$, so $D_{M_x}\subseteq V_x$.
		
		For the reverse inclusion, take $0\ne f\in V_x$. The set
		\[
		U=\{y\in X:v_y(f)\ge0\}
		\]
		is clopen and contains $x$. Let $e$ be its characteristic idempotent and put $h=T_{1-e}$. On $U$ we have $v_y(h)>0$ and $v_y(f)\ge0$; off $U$ we have $v_y(h)=0$ and $v_y(f)<0$. Consequently $c=1+hf$ is nonzero and
		\[
		v_y(c)=
		\begin{cases}0,&y\in U,\\v_y(f),&y\notin U.\end{cases}
		\]
		It follows that $s=1/c\in D$ and $fs\in D$: their values on $U$ are $0,v_y(f)$, and their values off $U$ are $-v_y(f),0$, respectively. Since $x\in U$, $v_x(s)=0$, so $s\notin M_x$. Thus $f=(fs)/s\in D_{M_x}$.
		
		The maximal ideal of $V_x$ consists of the positive-value elements. The least-positive-value argument in Proposition~\ref{kernel}, now at the single point $x$, identifies it with $tV_x$. Its quotient is its residue field $k_x$.
	\end{proof}
	
	\begin{proposition}\label{nonnoetherian}
		The particular domain $D$ constructed in this section, and every $D_{M_x}$, are non-Noetherian. In fact, if $u=T_1-1$, then
		\[
		u/t^n\in D\quad(n\ge0),\qquad
		uD\subsetneq (u/t)D\subsetneq (u/t^2)D\subsetneq\cdots.
		\]
	\end{proposition}
	
	\begin{proof}
		At every point, translation sends $T_1-1$ to $Y_1$. In positive characteristic its integer polynomial is primitive, so no Gauss value is introduced. Hence $v_x(u)=(0,\gamma_1)$. Since $1_R\ne0_R$, the $H$-component of $\gamma_1$ is positive, and
		\[
		\gamma_1>n\gamma_0\quad\text{for every }n\ge0.
		\]
		Thus $v_x(u/t^n)>0$ for all $x$, proving membership in $D$. Each displayed ideal is contained in the next because $u/t^n=t(u/t^{n+1})$. If equality held at one step, then $u/t^{n+1}=(u/t^n)d$ for some $d\in D$. Cancellation in the field $K$ would give $d=t^{-1}$, contradicting that $t$ is a nonunit. The same argument in $V_x$, where $t$ is also a nonunit, proves that each maximal localization is non-Noetherian.
	\end{proof}
	
	This non-Noetherianity is a feature of the present construction, not an obstruction attached to every target. Section~\ref{sec-semisimple} determines when a finite product of fields admits a different construction with a PID source.
	
	\begin{corollary}\label{kalgebra}
		If $R\ne0$ is an absolutely flat algebra over a field $k$, then there are a B\'ezout $k$-domain $D_k$ and a nonzero nonunit $t\in D_k$ such that, as $k$-algebras,
		\[
		D_k/tD_k\cong R,\qquad \Jac(D_k)=tD_k,\qquad \wgd D_k=1.
		\]
	\end{corollary}
	
	\begin{proof}
		A unital map from $k$ to the ring $R$ is injective. Its composite with each residue map also embeds $k$ into $k_x$. Work in $k(T_a:a\in R)$ and use the translated monomial valuations with values in $\Gamma$. Constants from $k^\times$ have value zero. For a nonzero polynomial over $k$, its image over each $k_x$ is nonzero, and translation remains an automorphism. The proof of Lemma~\ref{leadingpartition} therefore gives finite clopen leading-term partitions directly, with coefficients in $R$.
		
		For a fraction $P/Q$ integral at every point, the residue is the element
		\[
		\sum_\alpha e^P_\alpha e^Q_\alpha
		A_\alpha(P)\bigl(A_\alpha(Q)\bigr)^\#\in R.
		\]
		The equal- and unequal-leading-weight cases in Proposition~\ref{residuemap} verify this formula. Variables again give surjectivity. Propositions~\ref{bezout} and \ref{kernel} apply with least positive value $\gamma_0$, and the weak-dimension proof is unchanged. The maps fix $k$, so the resulting quotient is a $k$-algebra isomorphism.
	\end{proof}
	
	\begin{example}\label{products}
		For any nonempty family of fields $(k_i)_{i\in J}$, the ring $R=\prod_{i\in J}k_i$ is von Neumann regular: define $a^\#$ coordinatewise by inversion at the nonzero coordinates and by zero at the others. Theorem~\ref{mainA} therefore applies to arbitrary products, including products with different residue characteristics.
		
		One must use all of $\Max R$, not just the kernels of the coordinate projections. For example, fix a prime $p$ and let $R=\prod_{n\ge1}\F_{p^n}$. A nonprincipal ultrafilter $\mathcal U$ on $\N_{>0}$ gives a surjective map to the ultraproduct field and hence a maximal ideal. To see directly that the ultraproduct is a field, represent a nonzero class by a sequence nonzero on a set in $\mathcal U$, and invert on that set. Choose $a_n\in\F_{p^n}$ of degree $n$ over $\F_p$. If $0\ne F\in\F_p[Z]$, then $F(a_n)\ne0$ for $n>\deg F$. Thus the class of $(a_n)$ is transcendental over $\F_p$ in the ultraproduct. The construction makes no algebraicity assumption on these additional residue fields.
	\end{example}
	
	The weak-dimension dichotomy for these principal quotients is a special case of \cite[Theorem~4]{Sorokin}. For the powers of our distinguished element, the following direct calculation also identifies nonzero Tor groups in every positive degree.
	
	\begin{proposition}\label{higherpowers}
		For $n\ge2$, let $A_n=D/t^nD$. Then
		\[
		\Jac(A_n)=tD/t^nD,\qquad
		A_n/\Jac(A_n)\cong R,\qquad \wgd A_n=\infty.
		\]
		The nilpotence index of $\Jac(A_n)$ is exactly $n$.
	\end{proposition}
	
	\begin{proof}
		Since $t^nD\subseteq\Jac(D)=tD$, the maximal-ideal correspondence gives the radical and residue-quotient assertions. The radical has zero $n$th power. Its $(n-1)$st power is nonzero, since $t^{n-1}\in t^nD$ would imply $1\in tD$ by cancellation.
		
		Let $\tau=t+t^nD$ and $M=A_n/\tau A_n\cong R$. Cancellation in $D$ gives
		\[
		\operatorname{ann}_{A_n}(\tau)=\tau^{n-1}A_n,
		\qquad \operatorname{ann}_{A_n}(\tau^{n-1})=\tau A_n.
		\]
		For example, $td\in t^nD$ is equivalent to $d\in t^{n-1}D$; the other equality is proved in the same way. We therefore have the exact free resolution with alternating differentials
		\[
		\cdots\xrightarrow{\ \tau\ } A_n
		\xrightarrow{\ \tau^{n-1}\ } A_n
		\xrightarrow{\ \tau\ } A_n\longrightarrow M\longrightarrow0.
		\]
		After tensoring with $M$, every differential is zero. Thus
		$\Tor^{A_n}_i(M,M)\cong M\ne0$ for every $i\ge1$. This proves $\fd_{A_n}M=\infty$ and hence the weak-dimension assertion.
	\end{proof}
	
	\Needspace{22\baselineskip}
	\section{A characterization for semisimple rings}\label{sec-semisimple}
	
	Fix fields $k_1,\ldots,k_n$, with $n\ge1$, and put
	\[
	R=\prod_{i=1}^n k_i,\qquad \kappa_i=|k_i|,
	\qquad \mu=\max\{\aleph_0,\kappa_1,\ldots,\kappa_n\}.
	\]
	We first prove the cardinal bound needed for necessity. We then prepare the lifts used for sufficiency. The only existence theorem required for the coefficient rings is the existence of Cohen rings for arbitrary fields of positive characteristic.
	
	\subsection{The cardinal obstruction}
	
	The next bound is the Dedekind-domain case of Kearnes--Oman \cite[Lemma~2.1]{KearnesOman}. We give the shorter DVR proof in the form needed here.
	
	\begin{lemma}[Kearnes--Oman, Dedekind case]\label{cardinality}
		Let $A$ be a Dedekind domain and $\mathfrak m$ a nonzero maximal ideal. If $\kappa=|A/\mathfrak m|$, then
		\[
		|A|\le\kappa^{\aleph_0}.
		\]
		In particular, a finite residue field forces $|A|\le2^{\aleph_0}$.
	\end{lemma}
	
	\begin{proof}
		The ring $V=A_{\mathfrak m}$ is a DVR, with residue field $A/\mathfrak m$; see \cite[Tag 034X]{Stacks}. Let $\pi$ be its uniformizer. Multiplication by $\pi^r$ identifies $V/\pi V$ with $\pi^rV/\pi^{r+1}V$. It follows inductively from the quotient maps that
		\[
		|V/\pi^rV|=\kappa^r\qquad(r\ge1),
		\]
		where for infinite $\kappa$ the expression $\kappa^r$ is simply $\kappa$.
		Every nonzero element of $V$ has finite valuation, so $\bigcap_{r\ge1}\pi^rV=0$. Hence
		\[
		A\hookrightarrow V\hookrightarrow\prod_{r\ge1}V/\pi^rV.
		\]
		Taking cardinalities gives
		\[
		|A|\le\prod_{r\ge1}\kappa^r
		=\kappa^{\sum_{r\ge1}r}
		=\kappa^{\aleph_0}.
		\]
		For finite $\kappa\ge2$, the last cardinal is $2^{\aleph_0}$.
	\end{proof}
	
	Suppose a Dedekind domain $A$ maps onto $R$. If $A$ is a field, the nonzero quotient is itself a field, so $n=1$ and the cardinal conditions of Theorem~\ref{mainB} are automatic. Otherwise the kernels $\mathfrak m_i$ of the coordinate maps $A\to k_i$ are nonzero maximal ideals. Indeed, a zero kernel for one coordinate would identify $A$ with that field. Lemma~\ref{cardinality} gives
	\begin{equation}\label{necessary-bound}
		\kappa_j\le |A|\le\kappa_i^{\aleph_0}
		\qquad(1\le i,j\le n).
	\end{equation}
	This proves condition (4). Raising its inequalities to the power $\aleph_0$ gives
	\[
	\kappa_j^{\aleph_0}\le
	(\kappa_i^{\aleph_0})^{\aleph_0}
	=\kappa_i^{\aleph_0}.
	\]
	Interchanging $i,j$ proves condition (3). Conversely, (3) implies (4), since $\kappa_j\le\kappa_j^{\aleph_0}$.
	
	\subsection{Complete DVRs and independent lifts}
	
	\begin{lemma}\label{completeDVR}
		For every field $k$, there is a characteristic-zero complete DVR $C$ with residue field $k$. If $\pi$ is a uniformizer, then
		\[
		|C|=|k|^{\aleph_0},\qquad
		|a+\pi^rC|=|k|^{\aleph_0}
		\quad(a\in C,\ r\ge1).
		\]
	\end{lemma}
	
	\begin{proof}
		If $\Char k=0$, use $C=k[[u]]$ and $\pi=u$. If $\Char k=p>0$, use a Cohen ring with residue field $k$ and uniformizer $p$. This is the existence theorem for Cohen rings \cite{Cohen}; the precise arbitrary-residue-field statement is \cite[Tag 0328]{Stacks}. In particular, no perfectness assumption is imposed on $k$. It has characteristic zero, since $p$ is a nonzerodivisor and all integers prime to $p$ are units.
		
		Choose a set $S\subseteq C$ of representatives for $C/\pi C$, with $0\in S$. Every $c\in C$ has an expansion
		\[
		c=s_0+s_1\pi+s_2\pi^2+\cdots,\qquad s_j\in S.
		\]
		Existence follows by choosing the successive residues and dividing each error by $\pi$; completeness gives the limit. For uniqueness, if two digit sequences first differ at index $r$, their difference has valuation exactly $r$, so their sums cannot agree. Conversely, every digit sequence converges. Thus $C$ is in bijection with $S^{\N}$, and $|C|=|k|^{\aleph_0}$. Multiplication by $\pi^r$ followed by translation by $a$ is a bijection from $C$ onto $a+\pi^rC$, proving the second assertion.
	\end{proof}
	
	\begin{lemma}\label{independent-lifts}
		Let $C$ be a characteristic-zero complete DVR, let $\pi$ be a uniformizer, and put $M=|C|$. Let $\lambda$ be an infinite cardinal with $\lambda\le M$. For each $\alpha<\lambda$, prescribe a nonempty ball
		\[
		B_\alpha=a_\alpha+\pi^{r_\alpha}C,
		\qquad a_\alpha\in C,\quad r_\alpha\ge1.
		\]
		There are $x_\alpha\in B_\alpha$ such that the family $(x_\alpha)_{\alpha<\lambda}$ is algebraically independent over $\Q$.
	\end{lemma}
	
	\begin{proof}
		By Lemma~\ref{completeDVR} and its digit calculation, $M\ge2^{\aleph_0}>\aleph_0$, and every $B_\alpha$ has cardinality $M$. Suppose the choices have been made for $\beta<\alpha$. The field
		\[
		F_\alpha=\Q(x_\beta:\beta<\alpha)
		\]
		has cardinality at most $\max\{\aleph_0,|\alpha|\}$. The set of elements of $\Frac(C)$ algebraic over $F_\alpha$ has at most that cardinality: there are at most that many polynomials, and each nonzero polynomial has finitely many roots. Since $\alpha<\lambda\le M$ and $M$ is uncountable,
		\[
		\max\{\aleph_0,|\alpha|\}<M.
		\]
		Choose $x_\alpha\in B_\alpha$ outside this algebraic set. Transfinite recursion gives the desired family. At a successor stage the new element is transcendental over its predecessors; at a limit stage every polynomial relation would already involve finitely many earlier choices. Thus the entire family is algebraically independent. The argument uses no regularity assumption on the cardinal $M$.
	\end{proof}
	
	\subsection{A finite intersection with prescribed residues}
	
	We isolate the ideal calculation. It explains why the construction below gives a PID, rather than merely an intersection of valuation rings.
	
	\begin{lemma}\label{finite-intersection}
		Let $K$ be a field with valuations $v_i:K^\times\to\Z$, $1\le i\le n$, whose valuation rings $V_i$ have residue fields $k_i$. Suppose there are $z_1,\ldots,z_n\in K^\times$ with
		\begin{equation}\label{separating-values}
			v_i(z_j)=\begin{cases}1,&i=j,\\0,&i\ne j.\end{cases}
		\end{equation}
		Put $D=\bigcap_i V_i$, and suppose the product residue map
		\[
		\rho:D\longrightarrow\prod_i k_i
		\]
		is surjective. Then $D$ is a semilocal PID. If $t=\prod_i z_i$ and $\mathfrak n_i=\{d\in D:v_i(d)>0\}$, then
		\[
		\ker\rho=tD=\Jac(D),\qquad
		\Max D=\{\mathfrak n_1,\ldots,\mathfrak n_n\},
		\qquad D_{\mathfrak n_i}=V_i,
		\]
		and $\Frac(D)=K$.
	\end{lemma}
	
	\begin{proof}
		Each $z_j$ belongs to $D$, and $v_i(t)=1$ for every $i$. For nonzero $d\in D$, its residue is zero at every place precisely when $v_i(d)\ge1$ for every $i$. This is equivalent to $d/t\in D$. Hence $\ker\rho=tD$.
		
		An element of $D$ is a unit precisely when all its values are zero: in that case its inverse lies in every $V_i$ and hence in $D$. For every $a\in D$, the element $1-ta$ has value zero everywhere, so it is a unit. Therefore $tD\subseteq\Jac(D)$. Since $D/tD\cong\prod_i k_i$ has zero Jacobson radical, equality follows. The maximal-ideal correspondence now gives exactly the $n$ distinct ideals $\mathfrak n_i$ in the statement.
		
		Take a nonzero ideal $J$ of $D$. For each $i$, the set of values of its nonzero elements is a nonempty subset of $\N$. Let
		\[
		m_i=\min\{v_i(f):0\ne f\in J\},\qquad
		b=\prod_i z_i^{m_i}.
		\]
		Then $v_i(b)=m_i$, so $f/b\in D$ for every $f\in J$. Consequently $L=b^{-1}J$ is an ideal of $D$. Choose $f_i\in J$ with $v_i(f_i)=m_i$. The element $f_i/b$ has $i$-th value zero, so $L\nsubseteq\mathfrak n_i$. As these are all maximal ideals, $L=D$. Thus $J=bD$. The zero ideal is principal as well, and $D$ is a PID.
		
		For $f\in K^\times$, choose $N\ge0$ such that $v_i(f)+N\ge0$ for every $i$. Then $ft^N\in D$, proving $K=\Frac(D)$. Finally, $D_{\mathfrak n_i}\subseteq V_i$, because all denominators outside $\mathfrak n_i$ have $i$-th value zero. For $f\in V_i\setminus\{0\}$ set
		\[
		s=\prod_{j\ne i}z_j^{\max\{0,-v_j(f)\}}.
		\]
		Equation~\eqref{separating-values} gives $s\in D\setminus\mathfrak n_i$ and $sf\in D$. Hence $f=(sf)/s\in D_{\mathfrak n_i}$. This proves the localization assertion.
	\end{proof}
	
	\begin{proof}[Proof of Theorem~\ref{mainB}]
		The implication (1)$\Rightarrow$(4), and the equivalence of (3) and (4), were proved in \eqref{necessary-bound} and the following paragraph. Condition (2) implies (1), since a PID is Dedekind. It remains to construct the source under condition (3), including the stated cardinality.
		
		Write $M=\kappa_i^{\aleph_0}$ and fix any cardinal $\lambda$ with $\mu\le\lambda\le M$. For each $i$, choose a characteristic-zero complete DVR $C_i$ with uniformizer $\pi_i$ and residue field $k_i$, as in Lemma~\ref{completeDVR}. Thus $|C_i|=M$ for every $i$.
		
		Let $\mathcal U$ be a set of $\lambda$ distinct variables containing the distinguished variables
		\[
		\{T_a:a\in R\}\quad\text{and}\quad Z_1,\ldots,Z_n.
		\]
		Such a set exists because $|R|\le\mu\le\lambda$; for a finite product of infinite or finite sets this follows directly from cardinal arithmetic. The two displayed families are disjoint. Add extra variables if necessary, and put $K=\Q(\mathcal U)$.
		
		For $a=(a_1,\ldots,a_n)\in R$, prescribe that the image of $T_a$ in $C_i$ belong to the residue class of $a_i$. For the distinguished $Z_j$, prescribe the balls
		\begin{equation}\label{prescribed-balls}
			Z_j\longmapsto
			\begin{cases}
				\pi_i+\pi_i^2C_i,&j=i,\\
				1+\pi_iC_i,&j\ne i.
			\end{cases}
		\end{equation}
		For an extra variable, prescribe $\pi_iC_i$. Each requirement is a ball of the form in Lemma~\ref{independent-lifts}. Apply that lemma separately for each $i$, after well-ordering $\mathcal U$ with order type $\lambda$. The chosen images of the variables are algebraically independent over $\Q$. They therefore define an embedding
		\[
		\sigma_i:K\hookrightarrow\Frac(C_i).
		\]
		The variables lift the target elements in each coefficient DVR, and the algebraic independence gives one common rational function field for the pulled-back valuations.
		
		Let $v_i$ be the normalized valuation of $C_i$ pulled back along $\sigma_i$, and let $V_i$ be its valuation ring in $K$. All variables have nonnegative value. The prescriptions \eqref{prescribed-balls} give
		\[
		v_i(Z_j)=\begin{cases}1,&i=j,\\0,&i\ne j,\end{cases}
		\]
		so every $v_i$ has value group exactly $\Z$. A valuation ring with value group $\Z$ is a DVR \cite[Section~10.50]{Stacks}. Thus each $V_i$ is a DVR. Its residue field is $k_i$: it embeds in the residue field of $C_i$, and the residues of the variables $T_a$ contain all elements of $k_i$.
		
		Now set $D=\bigcap_iV_i$. Every variable belongs to $D$, and the product residue map satisfies
		\[
		\rho(T_a)=a\qquad(a\in R).
		\]
		Thus it is surjective onto the specified product of fields. Lemma~\ref{finite-intersection}, with $z_i=Z_i$, proves that $D$ is a semilocal PID with exactly $n$ maximal ideals and that
		\[
		D/tD\cong R,\qquad \Jac(D)=tD,
		\qquad t=Z_1\cdots Z_n.
		\]
		The domain has characteristic zero since $D\subseteq K$. The element $t$ is nonzero and has value one at every place, so it is a nonunit. Finally,
		\[
		\Z[\mathcal U]\subseteq D\subseteq\Q(\mathcal U),
		\qquad \Frac(D)=K.
		\]
		Both outer rings have cardinality $\lambda$, because $\lambda$ is infinite. Hence $|D|=\lambda$, completing the proof.
	\end{proof}
	
	\subsection{Consequences and examples}
	
	\begin{corollary}\label{source-sizes}
		Assume the equivalent conditions of Theorem~\ref{mainB}. With $M$ and $\mu$ as there, the possible cardinalities of nonfield Dedekind domains mapping onto $R$ are exactly the cardinals in the interval
		\[
		\mu\le\lambda\le M.
		\]
		Every one of these cardinalities is attained by a characteristic-zero semilocal PID whose Jacobson radical is the quotient kernel.
	\end{corollary}
	
	\begin{proof}
		A nonfield integral domain is infinite. A surjection onto each $k_i$ gives the lower bound $|D|\ge\mu$, while Lemma~\ref{cardinality} applied to each coordinate kernel gives $|D|\le M$. The construction in Theorem~\ref{mainB} gives every cardinal in this interval.
	\end{proof}
	
	\begin{corollary}\label{smallfields}
		Every finite product of fields of cardinality at most $\mathfrak c=2^{\aleph_0}$ is a quotient of a characteristic-zero semilocal PID by its principal Jacobson radical. If one factor of $R$ is finite, this holds if and only if every factor has cardinality at most $\mathfrak c$.
	\end{corollary}
	
	\begin{proof}
		For every cardinal $2\le\kappa\le\mathfrak c$,
		\[
		\mathfrak c=2^{\aleph_0}
		\le\kappa^{\aleph_0}
		\le\mathfrak c^{\aleph_0}=\mathfrak c.
		\]
		Thus the cardinal powers in Theorem~\ref{mainB} are equal. If one factor is finite, its countable power is $\mathfrak c$. Equality of all the powers then implies $\kappa_i\le\kappa_i^{\aleph_0}=\mathfrak c$ for every $i$. The converse is the first assertion.
	\end{proof}
	
	\begin{example}
		Both $\F_2\times\mathbb C$ and $\Q\times\F_3\times\mathbb C$ have characteristic-zero semilocal PID sources. No common coefficient field is required. More generally, fields of the same cardinality always satisfy Theorem~\ref{mainB}, even when that cardinality exceeds the continuum. The condition concerns their countable powers, so equality of the field cardinalities themselves is sufficient but is not required.
	\end{example}

	\begin{corollary}\label{negative-example}
		If $|S|>\mathfrak c$ and
		\[
		F=\F_2(X_s:s\in S),\qquad R=\F_2\times F,
		\]
		then $\gld R=0$, but $R$ is not a quotient of any Dedekind domain. It is not a quotient of any integral domain of global dimension at most one. Nevertheless, it is a quotient of a characteristic-zero B\'ezout domain of weak global dimension one by its principal Jacobson radical.
	\end{corollary}
	
	\begin{proof}
		Since every rational function uses finitely many indices, $|F|=|S|>\mathfrak c$. But $|\F_2|^{\aleph_0}=\mathfrak c$, so condition (4) of Theorem~\ref{mainB} fails. This excludes every Dedekind source, regardless of the kernel.
		
		The equality $\gld R=0$ is the standard semisimple-ring criterion \cite[Section~4.1]{Rotman}. Explicitly, every $R$-module splits into its two idempotent components, which are vector spaces over $\F_2$ and $F$. Choosing bases writes it as a direct sum of copies of the projective modules $(1,0)R$ and $(0,1)R$. Hence every $R$-module is projective and $\gld R=0$. Finally, $R$ is von Neumann regular, so Theorem~\ref{mainA} supplies the stated B\'ezout source.
	\end{proof}

\end{document}